\documentclass[11pt]{amsart}
\usepackage{amssymb,amscd,amsthm,latexsym}
\allowdisplaybreaks[3]
\usepackage{graphicx}
\usepackage[british]{babel}
\usepackage[all]{xy}
\usepackage{color}
\usepackage{mathabx}
\usepackage{calrsfs}
\usepackage{booktabs}
\usepackage[mathscr]{euscript} 

\theoremstyle{plain}
\newtheorem{theorem}[subsection]{Theorem}
\newtheorem{lemma}[subsection]{Lemma}
\newtheorem{proposition}[subsection]{Proposition}
\newtheorem{corollary}[subsection]{Corollary}

\theoremstyle{definition}
\newtheorem{definition}[subsection]{Definition}
\newtheorem{example}[subsection]{Example}
\newtheorem{examples}[subsection]{Examples}

\theoremstyle{remark}

\newtheorem{remark}[subsection]{Remark}

\newdir{ >}{
 @{}*!/-10pt/@{>} }
\newdir{ |>}{
 @{}*!/-5.5pt/@{|}*!/-10pt/:(1,-.2)@^{>}*!/-10pt/:(1,+.2)@_{>} }

\newdir{|>}{!/4.5pt/@{|}*:(1,-.2)@^{>}*:(1,+.2)@_{>}}

\newdir{ >>}{{}*!/8pt/@{|}*!/3.5pt/:(1,-.2)@^{>}*!/3.5pt/:(1,+.2)@_{>}}

\newcommand{\defn}[1]{\emph{#1}}

\newcommand{\To}{\Rightarrow}
\newcommand{\caixa}[1]{\framebox{\scriptsize\ensuremath{\mathrm{#1}}}}

\def\mathrmdef#1{\expandafter\def\csname#1\endcsname{{\rm#1}}}
\mathrmdef{max}
\mathrmdef{op}
\mathrmdef{prod}
\mathrmdef{lex}
\mathrmdef{id}

\newcommand{\C}{\ensuremath{\mathsf{C}}}
\newcommand{\CC}{\ensuremath{\mathbb{C}}}

\newcommand{\D}{\ensuremath{\mathsf{D}}}

\newcommand{\Grp}{\ensuremath{\mathsf{Grp}}}
\newcommand{\TopGrp}{\ensuremath{\mathsf{TopGrp}}}
\newcommand{\PTopGrp}{\ensuremath{\mathsf{PTopGrp}}}
\newcommand{\Ord}{\ensuremath{\mathsf{Ord}}}

\newcommand{\Ab}{\ensuremath{\mathsf{Ab}}}
\newcommand{\Mon}{\ensuremath{\mathsf{Mon}}}

\newcommand{\Set}{\ensuremath{\mathsf{Set}}}
\newcommand{\Nil}{\ensuremath{\mathsf{Nil}}}
\newcommand{\NN}{\ensuremath{\mathbb{N}}}
\newcommand{\RR}{\ensuremath{\mathbb{R}}}
\newcommand{\ZZ}{\ensuremath{\mathbb{Z}}}
\newcommand{\Pt}{\ensuremath{\mathrm{Pt}}}
\newcommand{\Eq}{\ensuremath{\mathrm{Eq}}}

\newcommand{\M}{\ensuremath{\mathscr{M}}}
\newcommand{\N}{\ensuremath{\mathscr{N}}}
\newcommand{\Sscr}{\ensuremath{\mathscr{S}}}

\newcommand{\OrdGrp}{\ensuremath{\mathsf{OrdGrp}}}
\newcommand{\ROrdGrp}{\ensuremath{\mathsf{ROrdGrp}}}
\newcommand{\OOrdGrp}{\ensuremath{\mathbb{O}}\ensuremath{\mathsf{rdGrp}}}

\def\mathrmdef#1{\expandafter\def\csname#1\endcsname{{\rm#1}}}
\newcommand{\la}{\langle}
\newcommand{\ra}{\rangle}
\newcommand{\mono}{\rightarrowtail}

\def\vpullback{
\ar@{-}[]+D+<6pt,-6pt>;[]+D+<0pt,-12pt>;%
\ar@{-}[]+D+<0pt,-12pt>;[]+D+<-6pt,-6pt>}

\def\pullback{
 \ar@{-}[]+R+<4pt,-3pt>;[]+RD+<4pt,-6pt>%
 \ar@{-}[]+D+<3pt,-6pt>;[]+RD+<4pt,-6pt>}

\def\ophalfsplitpullback{%
 \ar@{-}[]+R+<6pt,-1pt>;[]+RD+<6pt,-6pt>%
 \ar@{-}[]+D+<.5ex,-6pt>;[]+RD+<6pt,-6pt>}

\def\pushout{%
 \ar@{-}[]+LU+<-5pt,0pt>;[]+LU+<-5pt,5pt>%
 \ar@{-}[]+LU+<0pt,5pt>;[]+LU+<-5pt,5pt>}

\def\ophalfsplitpushout{%
 \ar@{-}[]+L+<-9pt,1pt>;[]+LU+<-9pt,9pt>%
 \ar@{-}[]+U+<-.5ex,9pt>;[]+LU+<-9pt,9pt>}

\usepackage{chngcntr}
\counterwithin*{equation}{section}

\def\mathrmdef#1{\expandafter\def\csname#1\endcsname{{\rm#1}}}
\mathrmdef{max}
\mathrmdef{op}
\mathrmdef{Nat}
\mathrmdef{sym}

\newcommand{\bkR}{\mathbb{R}}
\newcommand{\bkN}{\mathbb{N}}
\newcommand{\bkZ}{\mathbb{Z}}

\newcommand{\Cat}{\ensuremath{\mathsf{Cat}}}
\newcommand{\Top}{\ensuremath{\mathsf{Top}}}

\def\relto{{\rightarrow\hspace*{-2.4ex}{\mapstochar}\hspace*{2.2ex}}}

\newcommand{\Mor}{\ensuremath{\mathscr{M}\mathsf{or}}}
\newcommand{\Mono}{\ensuremath{\mathscr{M}\mathsf{ono}}}
\newcommand{\SMono}{\ensuremath{\mathscr{SM}\mathsf{ono}}}
\newcommand{\Spl}{\ensuremath{\mathscr{S}\mathsf{pl}}}
\newcommand{\PSpl}{\ensuremath{\mathscr{PS}\mathsf{pl}}}
\newcommand{\Ker}{\ensuremath{\mathrm{Ker}}}

\newcommand{\ab}{\ensuremath{\mathrm{ab}}}
\newcommand{\coker}{\ensuremath{\mathrm{coker}}}

\begin{document}

\title{How protomodular is your favourite category?}

\author{Maria Manuel Clementino}
\address{CMUC, Department of Mathematics, University of
Coimbra, 3000-143 Coimbra, Portugal}\thanks{}
\email{mmc@mat.uc.pt}

\author{Diana Rodelo}
\address{Department of Mathematics, University of the Algarve, 8005-139 Faro, Portugal and Center for Research and Development in Mathematics and Applications (CIDMA), Department of Mathematics, University of Aveiro, 3810-193 Aveiro, Portugal}
\thanks{The authors acknowledge partial financial support by CMUC (doi.org/10.54499/UID/00324/2025), funded by the Portuguese
Foundation for Science and Technology (FCT), grants UID/00324/2025 and UID/PRR/00324/2025. The second author also
acknowledges financial support by CIDMA (https://ror.org/05pm2mw36), funded by the Portuguese Foundation for Science
and Technology (FCT, https://ror.org/00snfqn58), grants UID/04106/2025 (https://doi.org/10.54499/UID/04106/2025) and
UID/PRR/04106/2025.
}

\email{drodelo@ualg.pt}



\begin{abstract} In this work we investigate a relative notion of protomodularity. We say that a finitely complete category $\C$ satisfies the \defn{$\M$-version of protomodularity} when the change-of-base functors (concerning the fibration of points) are conservative with respect to a class of morphisms $\M$ of $\C$. If $\M$ is the class of all morphisms or the class of all monomorphisms of $\C$, then this notion is precisely that of an ``ordinary'' protomodular category. We analyse what conditions the class $\M$ should have to obtain the relative $\M$-versions of some well-known properties of protomodular categories, such as the Split Short Five Lemma (when $\C$ is also pointed) and their relation with Mal'tsev categories. We also give diverse examples of categories satisfying the $\M$-version of protomodularity for different choices of $\M$.
\end{abstract}

\maketitle

\section*{Introduction}

Protomodular categories were introduced in~\cite{NEKEAC} with the purpose to provide a good categorical setting to capture group-like features. The main example of a (pointed) protomodular category is $\Grp$. Any other variety (of universal algebras) with more algebraic structure on its objects is still protomodular. On the other hand, the algebraically weaker variety $\Mon$ of monoids is not protomodular (for more examples of protomodular categories see~\cite{NEKEAC, MCFPO,BB}).

A finitely complete category $\C$ is called \defn{protomodular}~\cite{NEKEAC} when the change-of-base functors (with respect to the fibration of points $\Pi\colon \Pt(\C)\to \C$) $h^*\colon \Pt_Y(\C)\to \Pt_Z(\C)$ are conservative for any morphism $h\colon Z\to Y$. When $\C$ is also pointed, we can reduce our attention to morphisms of the type $h=0_Y\colon 0\to Y$ and the conservativeness of the kernel functors $0_Y^*=\ker\colon \Pt_Y(\C)\to \C$ for any object $Y$. It easily follows that a pointed finitely complete category is protomodular if and only if the Split Short Five Lemma holds in $\C$ (see Theorem~\ref{thm: characterization of protomodularity})
\begin{equation}\label{diag: SS5L in introduction}
	\vcenter{\hbox{%
	\xymatrix@=25pt{ K' \ar@{|>->}[r]^-{\ker(f')} \ar[d]|-{\cong} & X' \ar@<2pt>[r]^-{f'} \ar[d]_-m & Y' \ar@<2pt>[l]^-{s'} \ar[d]|-{\cong} \\
		K \ar@{|>->}[r]_-{\ker(f)} & X \ar@<2pt>[r]^-f & Y, \ar@<2pt>[l]^-s}%
	}}
	\qquad\To\qquad m \text{ is an isomorphism,}
\end{equation}
which is well known to hold in the variety $\Grp$ of groups.

Since the change-of-base functors preserve finitely limits, any $h^*$ is conservative if and only if it is conservative on monomorphisms. More precisely, if $m$ is a monomorphism (in $\Pt_Y(\C)$) and $h^*(m)$ is an isomorphism (in $\Pt_Z(\C)$), then $m$ is an isomorphism (in $\Pt_Y(\C)$). As a consequence, we only need to check the validity of the Split Short Five Lemma when $m$ is a monomorphism.

The category $\Ord\Grp$ of (pre)ordered groups and continuous (=order-preserving) group homomorphisms is not protomodular, even though its object are groups. It was shown in \cite{CM-FM} that the Split Short Five Lemma fails in $\OrdGrp$. Still, $\OrdGrp$ satisfies a weak version of it: given the commutative diagram~\eqref{diag: SS5L in introduction}, then $m$ is an isomorphism provided that it is a strong monomorphism. This led us to study categories satisfying a relative notion of protomodularity, placing a class $\M$ of morphisms where all morphisms were considered in the notion of protomodular category. This gave rise to the notion of category satisfying an \defn{$\M$-version of protomodularity}, which we simply call \defn{$\M$-protomodular}, proposed in this paper (see Definition~\ref{def: M-proto}). This should not be confused with the notion of $\Sscr$-protomodularity introduced by Bourn, Martins-Ferreira, Montoli and Sobral in \cite{BMFMS}, where $\Sscr$ refers to a class of points (or split extensions).

The aim of this work is to investigate $\M$-protomodularity for different classes $\M$ of (mono)mor\-phi\-sms in $\C$, with particular interest in finding examples of categories that are $\M$-protomodular, where $\M$ is not the class of all (mono)morphisms (i.e. $\C$ is not protomodular in the usual sense). Although these categories are not ``truly'' protomodular, they still exhibit some sort of protomodular flavour. Of course, any finitely complete category $\C$ is always $\M$-protomodular when $\M$ is the class of isomorphisms of $\C$. If this is the only class for which it is $\M$-protomodular, then $\C$ is far from being protomodular. If $\C$ is $\M$-protomodular, with $\M$ being the class of all (mono)morphisms, then $\C$ is truly protomodular. This led us to the title of this work.

In Section \ref{sec: M-protomodularity} we translate to our relative context well-known characterisations of protomodular category, showing that their $\M$-versions are equivalent under very mild conditions on $\M$.

Sections \ref{sec: Exs} and \ref{sec: Spl-proto} are dedicated to special properties and examples of $\M$-protomodular categories. As expected, in most of our examples the class of strong monomorphisms plays a crucial role. Still, the more general setting described in Section \ref{sec: Spl-proto}, where $\M$ is the class $\Spl$ of split monomorphisms, turns out to be very interesting; in particular, it is shown that in $\Spl$-protomodular categories pullbacks reflect monomorphisms.

In this context it is natural to introduce an $\M$-version of Mal'tsev categories. That is the goal of Section \ref{sec: M-proto vs M-Mal'tsev}, where it is shown that, as in the classical setting, $\M$-protomodular categories are $\M$-Mal'tsev categories.

In the last section we analyse $\M$-protomodularity from the point of view of $\Ord$-enriched categories. In the examples studied throughout this paper it turns out that the class of full and faithful (ff) monomorphisms is not strong enough to provide relative protomodularity for the $\Ord$-enrichments studied. Hence it is an open question whether there is an interesting $\M$-protomodular $\Ord$-category for $\M$ the class of ff-monomorphisms.

\section{$\M$-protomodularity}\label{sec: M-protomodularity}

In this section we will consider $\C$ to be a finitely complete category, although the existence of pullbacks along split epimorphisms will suffice for most of the results. Given an object $Y$ of $\C$, a \emph{point over $Y$} is given by a pair of morphisms $(f,s)$, where $f\colon X\to Y$ is a split epimorphism with a chosen splitting $s\colon Y\to X$. We denote by $\Pt_Y(\C)$ the category of points over $Y$ in $\C$. Given a morphism $h\colon Z\to Y$, its corresponding \emph{change-of-base functor} $h^*$ is defined as

\[
\begin{array}{rcl}
	h^*\colon \Pt_Y(\C) & \longrightarrow & \Pt_Z(\C), \\
		\xymatrix@R=20pt@C=30pt{X' \ar@<2pt>[dr]^-{f'} \ar[dd]_-{m} \\
			& Y \ar@<2pt>[ul]^-{s'} \ar@<2pt>[dl]^-{s} \\
			X \ar@<2pt>[ur]^-{f}} & \longmapsto &
		\xymatrix@R=20pt@C=30pt{Z\times_Y X' \ar@<2pt>[dr]^-{\pi'_Z} \ar[dd]_-{h^*(m)} \\
			& Z, \ar@<2pt>[ul]^-{i'_Z} \ar@<2pt>[dl]^-{i_Z} \\
			Z\times_Y X \ar@<2pt>[ur]^-{\pi_Z}}
\end{array}
\]
where the front and back faces of the following diagram are pullbacks
\begin{equation}\label{diag: def of h*}
	\vcenter{\xymatrix@R=30pt@C=50pt{Z\times_Y X'  \ar[rr]^-{\pi_{X'}}  \ar@{.>}[dr]^-{h^*(m)} \ar@/_/@<-7pt>[ddr]_-{\pi'_Z} & & X' \ar@/_/@<-2pt>[ddr]_(.25){f'} \ar[dr]^-{m}  \\
		& Z\times_Y X \ar[rr]^(.4){\pi_{X}} \ar@<-2pt>[d]_-{\pi_Z} & & X \ar@<-2pt>[d]_(.3){f} \\
		& Z \ar@<-2pt>[u]_-{i_Z} \ar[rr]_-h \ar@/^/@<3pt>[uul]_-{i'_Z} & & Y; \ar@<-2pt>[u]_(.7){s} \ar@/^/@<-2pt>[uul]_(.65){s'}}}
\end{equation}
$i_Z=\la 1_Z,sh\ra$ and $i'_Z=\la 1_Z,s'h\ra$. Note that a morphism $m\colon (f',s')\to (f,s)$ is an isomorphism (resp. monomorphism) in $\Pt_Y(\C)$ if and only if $m\colon X'\to X$ is an isomorphism (resp. monomorphism) in $\C$.

We will be interested in studying the conservativeness of the change-of-base functors with respect to a certain class $\M$ of morphisms of $\C$. We say that $h^*$ is \emph{conservative on $\M$} when for any $m\colon (f',s')\to (f,s)$ in $\Pt_Y(\C)$ with $m\in\M$, if $h^*(m)$ is an isomorphism, then $m$ itself is an isomorphism. If $\M=\Mono$ is the class of all monomorphisms of $\C$, we say that $h^*$ is conservative on monomorphisms.  It is well-known that the conservativeness of a left exact functor is equivalent to its conservativeness on monomorphisms. This applies to the change-of-base functors. Indeed, it is easy to check that $h^*$ preserves finite limits (this follows from the fact that $\C$ admits pullbacks of split epimorphisms).

In the proofs concerning the conservativeness of $h^*$, we eventually assume that the isomorphism $h^*(m)$ is an identity morphism.

\begin{definition}\label{def: M-proto}
Let $\C$ be a finitely complete category and $\M$ a class of morphisms of $\C$. We say that $\C$ satisfies the \defn{$\M$-version of protomodularity} when, for any morphism $h\colon Z\to Y$, the change-of-base functor $h^*\colon \Pt_Y(\C)\to \Pt_Z(\C)$ is conservative on $\M$.
\end{definition}
For simplicity, we will often say instead that $\C$ is \defn{$\M$-protomodular}. This should not be confused with the well-known notion of $\Sscr$-protomodular category introduced in \cite{BMFMS}, where $\Sscr$ refers to a class of points of $\C$.

\begin{example}\label{ex: M-proto cats}
\begin{itemize}
	\item[1.] 
When $\M$ is the class $\Mor$ of all morphisms or the class $\Mono$ of all monomorphisms in $\C$, $\C$ satisfies the $\M$-version of protomodularity exactly when $\C$ is protomodular \cite{NEKEAC}.
	\item[2.] When $\M$ is the class $\SMono$ of all strong monomorphisms in $\C$, an $\M$-protomodular category is known as a \emph{weakly protomodular category}~\cite{CM-FM}.
\end{itemize}
	\item[3.] For examples concerning other classes $\M$ see Sections~\ref{sec: Exs} and~\ref{sec: Spl-proto}.
\end{example}

We want to explore the several known characterisations of protomodular categories in the context of $\M$-protomodular categories. Mainly, we are interested in the following ones (see~\cite{NEKEAC,BB}):

\begin{theorem} \label{thm: characterization of protomodularity}
Let $\C$ be a finitely complete category. The following conditions are equivalent:
\begin{itemize}
	\item[(i)] $\C$ is protomodular, i.e. the change-of-base functors $h^* \colon \Pt_Y(\C) \to \Pt_Z(\C)$ are conservative for every $h \colon Z \to Y$;
	\item[(ii)] if $\C$ is also pointed, the \emph{Split Short Five Lemma} holds: given a commutative diagram of split short exact sequences
	\begin{equation}\label{diag:iv}
	\vcenter{\xymatrix@=25pt{ K' \ar@{|>->}[r]^-{\ker(f')} \ar[d]|-{\cong} & X' \ar@<2pt>[r]^-{f'} \ar[d]_-m & Y' \ar@<2pt>[l]^-{s'} \ar[d]|-{\cong} \\
		K \ar@{|>->}[r]_-{\ker(f)} & X \ar@<2pt>[r]^-f & Y, \ar@<2pt>[l]^-s}
	}
	\end{equation}
	it follows that $m$ is also an isomorphism;
	\item[(iii)] for any pullback of a point $(f,s)$ along an arbitrary morphism
	\begin{equation}\label{diag:iii}
	\vcenter{\xymatrix@=25pt{ Z\times_Y X \ar@<-2pt>[d]_-{\pi_Z} \ar[r]^-{\pi_X} \ophalfsplitpullback & X \ar@<-2pt>[d]_-f  \\
							   Z \ar@<-2pt>[u]_-{i_Z} \ar[r]_-h & Y, \ar@<-2pt>[u]_-s }}
	\end{equation}
	the pair $(\pi_X,s)$ is jointly extremally epimorphic;
		\item[(iv)] for any commutative diagram involving points $(f,s)$ and $(g,t)$
	\begin{equation}\label{diag:ii}
	\vcenter{
	\xymatrix@R=25pt@C=35pt{ U \ar@<-2pt>[d]_-{g} \ar[r]^-{u} \ar@{}[dr]|-{\caixa{\mathrm{A}}} & X \ar@<-2pt>[d]_-{f} \ar[r]^-{x} \ar@{}[dr]|-{\caixa{\mathrm{B}}} & A 		\ar[d]^-q \\
	V \ar@<-2pt>[u]_-{t} \ar[r]_-v & Y \ar@<-2pt>[u]_-{s} \ar[r]_-y & B}}
	\end{equation}
	(also commuting with the upward sections) where \caixa{A} and \caixa{A}\caixa{B} are pullbacks, then \caixa{B} is also a pullback.
\end{itemize}
\end{theorem}
\begin{proof}
(i) $\Leftrightarrow$ (ii) $\Leftrightarrow$ (iii): see~\cite{NEKEAC, BB}.\\
(i) $\Leftrightarrow$ (iv): see~\cite{NEKEAC, BJ1998, BB}.
\end{proof}

When $\C$ is a pointed category, we say that the Split Sort Five Lemma holds \emph{with respect to $\M$} when the property stated in Theorem~\ref{thm: characterization of protomodularity}(ii) holds for $m\in \M$. Without loss of generality, and to simplify proofs, we shall consider the isomorphisms of the diagram of the Split Short Five Lemma to be the identity morphisms. As a consequence, in diagram~\eqref{diag: def of h*}, we take $Z\times_Y X'=Z\times_Y X$, $\pi'_Z=\pi_Z$ and $i'_Z=i_Z$.

The proofs of the $\M$-versions in the following results are similar to those for ``ordinary'' protomodular  categories.

\begin{proposition}\label{prop: M-proto iff SS5L for M} Let $\C$ be a pointed finitely complete category and $\M$ any class of morphisms of $\C$. The following conditions are equivalent:
\begin{itemize}
	\item[(i)] $\C$ is $\M$-protomodular;
	\item[(ii)] the Split Short Five Lemma holds with respect to $\M$.
\end{itemize}
\end{proposition}
\begin{proof} For $\M=\Mor$ or $\M=\Mono$ this result is the equivalence between (i) and (ii) of Theorem~\ref{thm: characterization of protomodularity}.\\
(i) $\Rightarrow$ (ii) We consider diagram~\eqref{diag:iv} with $K'=K$ and $Y'=Y$, so that the isomorphisms are the identity morphisms; suppose that $m\in\M$. This diagram shows that $0_Y^*(m)=1_K$ which, by assumption, implies that $m$ is an isomorphism.\\
(ii) $\Rightarrow$ (i) Consider an arbitrary morphism $h\colon Z\to Y$ of $\C$ and a morphism $m\colon X'\to X$ in $\M$ such that $h^*(m)=1_P$
\begin{equation}\label{diag: h*(m) iso}
	\vcenter{\xymatrix@R=30pt@C=50pt{P  \ar[rr]^-{\pi_{X'}}  \ar@{=}[dr]^-{\;h^*(m)=1_P} \ar@/_/@<-7pt>[ddr]_-{\pi_Z} & & X' \ar@/_/@<-2pt>[ddr]_(.25){f'} \ar[dr]^-{\;m\in \M}  \\
		& P \ar[rr]^(.4){\pi_{X}} \ar@<-2pt>[d]_-{\pi_Z} & & X \ar@<-2pt>[d]_(.3){f} \\
		& Z \ar@<-2pt>[u]_-{i_Z} \ar[rr]_-h \ar@/^/@<3pt>[uul]_-{i_Z} & & Y; \ar@<-2pt>[u]_(.7){s} \ar@/^/@<-2pt>[uul]_(.65){s'}}}
\end{equation}
here $P=Z\times_Y X=Z\times_Y X'$. By taking the kernel of $\pi_Z$ we get $\pi_{X'}\ker(\pi_Z)=\ker(f')$ and $\pi_{X}\ker(\pi_Z)=\ker(f)$
\[
\xymatrix@=25pt{ K \ar@{|>->}[r]^-{\ker(f')} \ar@{=}[d] & X' \ar@<2pt>[r]^-{f'} \ar[d]_-{m\in \M} & Y \ar@<2pt>[l]^-{s'} \ar@{=}[d] \\
		K \ar@{|>->}[r]_-{\ker(f)} & X \ar@<2pt>[r]^-f & Y, \ar@<2pt>[l]^-s}
\]
so that we can conclude that $m$ is an isomorphism, by assumption.
\end{proof}

\begin{proposition}\label{prop: M-proto iff orthogonal to M} Let $\C$ be a finitely complete category and $\M$ a class of monomorphisms of $\C$. The following conditions are equivalent:
\begin{itemize}
	\item[(i)] $\C$ is $\M$-protomodular;
	\item[(ii)] for any pullback of a point $(f,s)$ along an arbitrary morphism as in \eqref{diag:iii},	the pair $(\pi_X,s)$ is orthogonal to the class $\M$.
\end{itemize}
\end{proposition}
\begin{proof} For $\M=\Mono$ this result is the equivalence between (i) and (iii) of Theorem~\ref{thm: characterization of protomodularity}.\\
(i) $\Rightarrow$ (ii) To prove that $(\pi_X,s)$ is orthogonal to $\M$, we consider the commutative diagram
\[
	\xymatrix@=25pt{ & M \ar@{ >->}[d]|-{m\in \M} \\
		Z\times_Y X \ar[ur]^-a \ar[r]_-{\pi_X} & X & Y. \ar[ul]_-b \ar[l]^-s}
\]
This gives the commutative diagram
\[
	\xymatrix@R=30pt@C=50pt{Z\times_Y X  \ar[rr]^-{a}  \ar@{=}[dr]^-{\;h^*(m)=1_{Z\times_Y X}} \ar@/_/@<-7pt>[ddr]_-{\pi_Z} & & M \ar@/_/@<-2pt>[ddr]_(.25){fm} \ar@{ >->}[dr]^-{\;m\in \M}  \\
		& Z\times_Y X  \ar[rr]^(.4){\pi_{X}} \ar@<-2pt>[d]_-{\pi_Z} & & X \ar@<-2pt>[d]_(.3){f} \\
		& Z \ar@<-2pt>[u]_-{i_Z} \ar[rr]_-h \ar@/^/@<3pt>[uul]_-{i_Z} & & Y, \ar@<-2pt>[u]_(.7){s} \ar@/^/@<-2pt>[uul]_(.65){b}}
\]
where the front and top faces are pullbacks (the top face is a pullback since $m$ is a monomorphism). Note that $mai_Z =\pi_Xi_Z=sh=mbh$ implies that $a i_Z=bh$, since $m$ is a monomorphism. This proves that the back face commutes with respect to the sections. By assumption we may conclude that $m$ is an isomorphism.\\
(ii) $\Rightarrow$ (i) Consider diagram~\eqref{diag: h*(m) iso} with $m$ a monomorphism in $\M$. We get
\[
	\xymatrix@=25pt{ & X' \ar@{ >->}[d]|-{m\in \M} \\
		P \ar[ur]^-{\pi_X'} \ar[r]_-{\pi_X} & X & Y; \ar[ul]_-{s'} \ar[l]^-s}
\]
by assumption, $m$ is an isomorphism.
\end{proof}

\begin{proposition}\label{prop: pb pp} Let $\C$ be a finitely complete category and $\M$ a class of morphisms of $\C$ which is left-cancellative, that is, $ba\in \M \Rightarrow a\in \M$. The following conditions are equivalent:
\begin{itemize}
	\item[(i)] $\C$ is $\M$-protomodular;
	\item[(ii)] for any commutative diagram involving points $(f,s)$ and $(g,t)$
	\[\xymatrix@R=25pt@C=35pt{ U \ar@<-2pt>[d]_-{g} \ar[r]^-{u} \ar@{}[dr]|-{\caixa{\mathrm{A}}} & X \ar@<-2pt>[d]_-{f} \ar[r]^-{m\in \M} \ar@{}[dr]|-{\caixa{\mathrm{B}}} & A 		\ar[d]^-q \\
	V \ar@<-2pt>[u]_-{t} \ar[r]_-v & Y \ar@<-2pt>[u]_-{s} \ar[r]_-y & B}
	\]
	(also commuting with the upward sections) where \caixa{A} and \caixa{A}\caixa{B} are pullbacks, then \caixa{B} is also a pullback.
\end{itemize}
\end{proposition}
\begin{proof} For $\M=\Mor$ or $\M=\Mono$ this result is the equivalence between (i) and (iv) of Theorem~\ref{thm: characterization of protomodularity}.\\
(i) $\Rightarrow$ (ii) Take the pullback of $(y,q)$ and the induced morphism $z$ in
\[
\xymatrix@R=25pt@C=35pt{X \ar@(r,u)[drr]^-{m\in \M} \ar@{.>}[dr]^-z \ar@(d,l)@<-2pt>[ddr]_-f  \\
	& Z \ar@{}[dr]|-{\caixa{\mathrm{C}}} \ar@<-2pt>[d]_-{\pi_Y} \ar[r]^-{\pi_A} \ophalfsplitpullback & A \ar[d]^-q \\
	& Y \ar@<-2pt>@(l,d)[uul]_-s \ar@<-2pt>[u]_-{zs} \ar[r]_-y & B.}
\]
Note that $z\in \M$, because $\M$ is left-cancellative. We get
\[
\xymatrix@R=25pt@C=35pt{U \ar[r]^-{zu} \ar@<-2pt>[d]_-g \ar@{}[dr]|-{\caixa{\mathrm{D}}} & Z \ar@{}[dr]|-{\caixa{\mathrm{C}}} \ophalfsplitpullback \ar@<-2pt>[d]_-{\pi_Y} \ar[r]^-{\pi_A} & A \ar[d]^-q \\
	V \ar@<-2pt>[u]_-t \ar[r]_-v & Y \ar@<-2pt>[u]_-{zs} \ar[r]_-y & B,}
\]
where \caixa{D}\caixa{C}=\caixa{A}\caixa{B} and \caixa{C} are pullbacks, so that \caixa{D} is a pullback. This gives the following commutative diagram
\[
	\xymatrix@R=30pt@C=50pt{U  \ar[rr]^-{u}  \ar@{=}[dr]^-{\;v^*(z)=1_U} \ar@/_/@<-7pt>[ddr]_-{g} & & X \ar@/_/@<-2pt>[ddr]_(.25){f} \ar[dr]^-{\;z\in \M}  \\
		& U \ophalfsplitpullback \ar[rr]^(.4){zu} \ar@<-2pt>[d]_-{g} & & Z \ar@<-2pt>[d]_(.3){\pi_Y} \\
		& V \ar@<-2pt>[u]_-{t} \ar[rr]_-v \ar@/^/@<3pt>[uul]_-{t} & & Y, \ar@<-2pt>[u]_(.7){zs} \ar@/^/@<-2pt>[uul]_(.65){s}}
\]
from which we deduce that $z$ is an isomorphism; consequently \caixa{B} is a pullback.\\
(ii) $\Rightarrow$ (i) Consider diagram~\eqref{diag: h*(m) iso}. We get
\[
	\xymatrix@R=25pt@C=35pt{ P \ar@(u,u)[rr]^-{\pi_X} \ar@<-2pt>[d]_-{\pi_Z} \ar[r]^-{\pi_{X'}} \ar@{}[dr]|-{\caixa{\mathrm{A}}} & X' \ar@<-2pt>[d]_-{f'} \ar[r]^-{m\in \M} \ar@{}[dr]|-{\caixa{\mathrm{B}}} & X \ar[d]^-f \\
	Z \ar@(d,d)[rr]_-{h} \ar@<-2pt>[u]_-{i_Z} \ar[r]_-h & Y \ar@<-2pt>[u]_-{s'} \ar@{=}[r]_-{1_Y} & Y,}
\]
where \caixa{A} and \caixa{A}\caixa{B} are pullbacks, so that \caixa{B} is also a pullback; thus $m$ is an isomorphism.
\end{proof}

\begin{proposition}\label{prop: pb pp for regular cat} Let $\C$ be a regular category and $\M$ a class of morphisms of $\C$ with the properties
\begin{itemize}
	\item[(P1)] $\M$ is left-cancellative;
	\item[(P2)] $\M$ is closed under composition;
	\item[(P3)] $m\in \M \Rightarrow m\times m \in \M$;
	\item[(P4)] for any pullback projections $\pi_Z$, $\pi_X$, the monomorphism $\la \pi_Z,\pi_X\ra\colon Z\times_Y X\mono Z\times X$ is in $\M$.
\end{itemize}
The following conditions are equivalent:
\begin{itemize}
	\item[(i)] $\C$ is $\M$-protomodular;
	\item[(ii)] for any commutative diagram
	\[\xymatrix@R=25pt@C=35pt{ U \ar@{>>}[d]_-{g} \ar[r]^-{u} \ar@{}[dr]|-{\caixa{\mathrm{A}}} & X \ar@{>>}[d]_-{f} \ar[r]^-{m\in \M} \ar@{}[dr]|-{\caixa{\mathrm{B}}} & A 		\ar[d]^-q \\
	V \ar[r]_-v & Y \ar[r]_-y & B}
	\]
	where $f,g$ are regular epimorphisms, if \caixa{A} and \caixa{A}\caixa{B} are pullbacks, then \caixa{B} is also a pullback.
\end{itemize}
\end{proposition}
\begin{proof} For $\M=\Mor$ or $\M=\Mono$ the proof of the equivalence between (i) and (ii) can be found in~\cite{NEKEAC}.\\
(i) $\Rightarrow$ (ii) Take vertical kernel pairs as in
	\[
	\xymatrix@R=30pt@C=35pt{ \Eq(g) \ar@{}[dr]|-{\caixa{\mathrm{C}}} \ar[r]^-{\overline{u}} \ar@<-4pt>[d]_-{g_1} \ar@<4pt>[d]^-{g_2} & \Eq(f) \ar@{}[dr]|-{\caixa{\mathrm{D}}} \ar[r]^-{\overline{m}} \ar@<-4pt>[d]_-{f_1} \ar@<4pt>[d]^-{f_2} & \Eq(q) \ar@<-4pt>[d]_-{q_1} \ar@<4pt>[d]^-{q_2}\\
	U \ar[u]  \ar@{>>}[d]_-{g} \ar[r]^-{u} \ar@{}[dr]|-{\caixa{\mathrm{A}}} & X \ar[u] \ar@{>>}[d]_-{f} \ar[r]^-{m\in \M} \ar@{}[dr]|-{\caixa{\mathrm{B}}} & A \ar[u]		\ar[d]^-q \\
	V \ar[r]_-v & Y \ar[r]_-y & B.}
	\]
Since \caixa{A} and \caixa{A}\caixa{B} are pullbacks, then so are \caixa{C} and \caixa{C}\caixa{D}. From properties (P2), (P3) and (P4), $(m\times m) \la f_1,f_2\ra=\la q_1,q_2\ra \overline{m}\in \M$; thus $\overline{m}\in \M$ by property (P1). We can apply Proposition~\ref{prop: pb pp} to \caixa{C}\caixa{D}, to conclude that \caixa{D} is a pullback. By the Barr-Kock Theorem~\cite{Barr} we conclude that \caixa{B} is a pullback as well.\\
(ii) $\Rightarrow$ (i) Follows immediately from Proposition~\ref{prop: pb pp}.
\end{proof}

\begin{proposition}\label{prop: regular M-proto iff S5L for M} Let $\C$ be a pointed regular category and $\M$ a class of morphisms of $\C$ with the properties
\begin{itemize}
	\item[(P1)] $\M$ is left-cancellative;
	\item[(P2)] $\M$ is closed under composition;
	\item[(P3)] $m\in \M \Rightarrow m\times m \in \M$;
	\item[(P4)] for any pullback projections $\pi_Z$, $\pi_X$, the monomorphism $\la \pi_Z,\pi_X\ra\colon Z\times_Y X\mono Z\times X$ is in $\M$.
\end{itemize} The following conditions are equivalent:
\begin{itemize}
	\item[(i)] $\C$ is $\M$-protomodular;
	\item[(ii)] the Short Five Lemma holds with respect to $\M$, i.e. given a commutative diagram of short exact sequences,
	\[\xymatrix@=25pt{ K \ar@{|>->}[r]^-{\ker(f')} \ar@{=}[d] & X' \ar@{|>}[r]^-{f'} \ar[d]^-{m\in \M} & Y \ar@{=}[d] \\
		K \ar@{|>->}[r]_-{\ker(f)} & X \ar@{|>}[r]_-f & Y,}\]
		$m$ is necessarily an isomorphism.
\end{itemize}
\end{proposition}
\begin{proof} For $\M=\Mor$ or $\M=\Mono$ this result is the equivalence between (i) and (iv) can be found in~\cite{NEKEAC, BJ1998, BB}.\\
(i) $\Rightarrow$ (ii) We get the following commutative diagram
\[
	\xymatrix@R=25pt@C=35pt{
		K \ar@{|>}[d] \ar@{|>->}[r]^-{\ker(f')} \ar@{}[dr]|-{\caixa{A}} \ar@(u,u)[rr]^-{\ker(f)} & X' \ar[r]^-{m\in \M} \ar@{|>}[d]^-{f'} \ar@{}[dr]|-{\caixa{B}} & X \ar@{|>}[d]^-f \\
		0 \ar@{|>->}[r] & Y \ar@{=}[r] & Y,}
\]
where \caixa{A} and \caixa{A}\caixa{B} are pullbacks. By Proposition~\ref{prop: pb pp for regular cat} we conclude that \caixa{B} is a pullback, thus $m$ is an isomorphism.\\
(ii) $\Rightarrow$ (i) Follows from Proposition~\ref{prop: M-proto iff SS5L for M}.
\end{proof}

\begin{remark} In the following sections we are mainly interested in studying $\M$-protomodularity for the classes of monomorphisms $\M=\SMono$, $\M=\Spl$ the class of split monomorphisms and $\M=\PSpl$ the class of protosplit monomorphisms (=kernels of split epimorphisms). $\SMono$ has the properties (P1)-(P4), so any of the above equivalences hold for $\SMono$-protomodular categories, as long as $\C$ satisfies their initial assumptions. $\Spl$ has properties (P1)-(P3) but not (P4), so only the equivalences in Propositions~\ref{prop: M-proto iff SS5L for M}, \ref{prop: M-proto iff orthogonal to M} and \ref{prop: pb pp} hold for $\Spl$-protomodular (pointed) categories. $\PSpl$ only satisfies (P2); thus only the equivalences in Propositions~\ref{prop: M-proto iff SS5L for M} and \ref{prop: M-proto iff orthogonal to M} hold for $\Spl$-protomodular pointed categories.
\end{remark}

\section{Examples of $\M$-protomodular categories}\label{sec: Exs}

We shall use the following result to obtain several interesting examples of $\M$-protomodular categories.

\begin{theorem}\label{thm: F:CC->DD} Consider finitely complete categories $\C$ and $\D$, with classes of morphisms $\M$ and $\N$, respectively. Let $F\colon \C\to \D$ be a  functor which preserves pullbacks, is conservative on $\M$ and such that $F(\M)\subseteq\N$. If $\D$ is an $\N$-protomodular category, then $\C$ is an $\M$-protomodular category.
\end{theorem}
\begin{proof} The proof is a straightforward adaptation of the well-known result concerning  ``ordinary'' protomodularity~\cite{NEKEAC}.
\end{proof}

\begin{corollary}\label{cor:top}
Let $F\colon \C\to \D$ be a topological functor between finitely complete categories, and $\M$ a class of $F$-initial morphisms of $\C$. Then $\C$ is $\M$-protomodular provided that $\D$ is $F(\M)$-protomodular.

In particular, if $\D$ is protomodular, then $\C$ is $\SMono$-protomodular.
\end{corollary}
\begin{proof}
A topological functor preserves limits and colimits, it is conservative on initial morphisms (see \cite{ST}), and every strong monomorphism is initial. Hence the statements of our corollary follow from the theorem above.
\end{proof}

\begin{examples}\label{exe:Ord}
It is well-known that:
\begin{itemize}
\item $\Set^\op$ is a (co)complete protomodular category~\cite{NEKEAC,BB};
\item the forgetful functors $\Top\to\Set$, $\Ord\to\Set$, and more generally $V\mbox{-}\Cat\to\Set$ (see \cite{CT}) are topological, therefore also $\Top^\op\to\Set^\op$, $\Ord^\op\to\Set^\op$, $(V\mbox{-}\Cat)^\op\to\Set^\op$ are topological since the notion of topological functor is self-dual. (Here $\Top$ denotes the category of topological spaces and continuous maps, $\Ord$ denotes the category of preordered sets and monotone maps, and $V\mbox{-}\Cat$ denotes the category of $V$-categories and $V$-functors, for a commutative and unital quantale $V$.)
\end{itemize}
Hence Corollary \ref{cor:top} guarantees that $\Ord^\op,\;\Top^\op,\;(V\mbox{-}\Cat)^\op$ are $\SMono$-protomodular.

None of the categories $\Top^\op$, $\Ord^\op$, and consequently $(V\mbox{-}\Cat)^\op$ in general, are protomodular.
To prove this claim we give an example of a non-strong monomorphism $m$ with $h^*(m)$ isomorphism. We describe our diagram in $\Ord$, for simplicity. In the following diagram
\[
\xymatrix@=30pt{A \ar[r]^-h \ar@{ >->}@<-2pt>[d]_-\alpha & C \ar@{ >->}@<-2pt>[d]_-{\alpha'} \\ B  \ar[r]_-g \ar[d]_-m \ar@{>>}@<-2pt>[u]_-\beta & D \ophalfsplitpushout \ar@{>>}@<-2pt>[u]_-{\beta'} \ar[d]|-\cong \\ B' \ar[r] & D' }
\]
$A=\{x_1,x_2,y_1,y_2\}$ with $x_1\ge x_2\le y_2\ge y_1$,
$C=\{1,x_2,y_2\}$ with $1\geq x_2\le y_2\ge 1$, $B$ and $B'$ have as
underlying set $\{x_1,x_2,y_1,y_2,x_0, y_0\}$, the order on $B$ is given
by the order on $A$ together with $x_0\le x_1,\,y_1\le y_0$, and the
order on $B'$ has the extra relation $x_0\le y_0$. The morphism $h\colon
A\to C$ is defined by $h(x_1)=h(y_1)=1$, $h(x_2)=x_2$, $h(y_2)=y_2$. The
inclusions $\alpha\colon A\to B$, $\alpha'\colon C\to D$ are split monomorphisms,
with retractions $\beta\colon B\to A$, $\beta'\colon D\to C$ where both identify
$x_0$ with $x_1$ and $y_0$ with $y_1$. Now it is easy to check that the
bijection $m\colon B\to B'$ is not an isomorphism, although its pushout
along $g$ is an isomorphism, since the identification of $x_1$ and $y_1$
by $h$ implies that $x_0\le y_0$ in $D$.
\end{examples}

\begin{examples}\label{exe:Grp}
Analogously:
\begin{itemize}
\item $\Grp$ is a (co)complete protomodular category;
\item as shown in \cite{CM-FM, CM1, CM2} respectively, the forgetful functors $\OrdGrp\to\Grp$, $\ROrdGrp\to\Grp$, $V$-$\Grp\to\Grp$ are topological, where $\ROrdGrp$ denotes the category of right-ordered groups and monotone group homomorphisms, and $V\mbox{-}\Grp$ has as objects $V$-groups and as morphisms $V$-functors which are group homomorphisms, for a suitable quantale $V$ (as introduced in \cite{CM1}).
\end{itemize}
Hence, using again Corollary \ref{cor:top}, we conclude that the following categories are $\SMono$-pro\-to\-mo\-du\-lar:
\[\OrdGrp,\;\ROrdGrp,\;V\mbox{-}\Grp.\]
None of these categories is protomodular, as shown in \cite{CM-FM, CM1, CM2} respectively.
\end{examples}

\begin{example}\label{ex: paratopological groups}
It was shown in \cite{BC} that the category $\TopGrp$ of topological groups and continuous group homomorphisms is protomodular.  If in the definition of topological group we drop the continuity of the inversion map we end up with the notion of \emph{paratopological group}. Let $\PTopGrp$ be the category of paratopological groups and continuous group homomorphisms. It is a complete and cocomplete category, and the forgetful functor $\PTopGrp\to\Grp$ is topological; indeed the proof of topologicity of this functor uses the same arguments used to show that the forgetful functor $\TopGrp\to\Grp$ is topological. Hence, by Corollary \ref{cor:top}, $\PTopGrp$ is $\SMono$-protomodular.

However, in contrast to the case of $\TopGrp$, $\PTopGrp$ is not protomodular, showing that the inversion map of each group is essential for protomodularity.
Let us show that the Split Short Five Lemma fails.

Let $Y$ be the group $(\mathbb{R},+)$ equipped with the topology $\mathcal{T}_S$ of the Sorgenfrey line; that is, the topology having $\{[0,a[\,;\,a\in\bkR,\,a>0\}$ as basic neighbourhoods of $0$. It is a paratopological group but not a topological group.

In $(\bkR,+)\times(\bkR,+)$ we may consider two different topologies: the product topology $\mathcal{T}_1=\mathcal{T}_S\times\mathcal{T}_S$, and $\mathcal{T}_2$ having as basic neighbourhoods of $0$
\[\{U_a=\bkR\times\,]0,a[\,\cup\,[0,a[\times\{0\}\,;\,a\in\bkR,\,a>0\}.\]
It is straightforward to check that both make $\bkR\times\bkR$ a non-topological paratopological group.

Both $\mathcal{T}_1$ and $\mathcal{T}_2$ make the following split short exact sequence in $\Grp$
\[
	\xymatrix@=25pt{ Y \ar@{ >->}[r]^-{\la 1,0\ra}  & (\RR\times \RR, \mathcal{T}_i) \ar@<2pt>[r]^-{\pi_2}  & Y \ar@<2pt>[l]^-{\la 0,1\ra}}
\]
a split short exact sequence in $\PTopGrp$; indeed, as the product topology, the case of $\mathcal{T}_1$ clearly works. Let us check our claim for $\mathcal{T}_2$:
\begin{itemize}
\item $k=\langle 1,0\rangle$ is a strong monomorphism: $k^{-1}(U_a)=[0,a[$.
\item $s=\langle 0,1\rangle$ is continuous: $s^{-1}(U_a)=[0,a[$.
\item To show that $\pi_2$ is continuous it is enough to notice that $\pi_2$ is continuous in $0$, and this follows from $\pi_2(U_a)= [0,a[$.
\end{itemize}
Now, for each $a>0$, $U_a\supseteq [0,a[\times\,[0,a[$, hence $U_a$ is clearly a neighbourhood of $0$ in $\mathcal{T}_1$, which implies that the identity map $(\bkR\times\bkR,\mathcal{T}_1)\to(\bkR\times\bkR,\mathcal{T}_2)$ is continuous. Since it is not an isomorphism, this example shows that the Split Short Five Lemma fails.
\end{example}

\begin{example}\label{ex: Mon with M=isos and maps in Grp} In $\Mon$, consider the class of morphisms $\M$ to be all isomorphisms in $\Mon$ and all group homomorphisms. Then $\Mon$ is $\M$-protomodular. Indeed, consider diagram~\eqref{diag: def of h*} with $m$ a group homomorphism and $h^*(m)$ an isomorphism. Since $f$ is surjective, then $Y$ is also a group, so that the right face of this diagram lives in $\Grp$. The proof that $m$ is an isomorphism follows as in $\Grp$, since the only inverses we must take are in $X$ or $X'$.
\end{example}

\section{$\Spl$-protomodular categories}\label{sec: Spl-proto}
Consider the class $\Spl$ of split monomorphisms of $\C$. Every $\SMono$-protomodular category is in particular a $\Spl$-protomodular category. Next we will present an example of an $\Spl$-protomodular category which is not $\SMono$-protomodular. Before this example we note that:

\begin{proposition}\label{prop: Spl-proto => h* reflects monos}
If $\C$ is $\Spl$-protomodular, then every change-of-base functor $h^*$ reflects monomorphisms.
\end{proposition}
\begin{proof}
Consider diagram~\eqref{diag: def of h*} such that $h^*(m)$ is a monomorphism. We extend the diagram by taking kernel pairs of $h^*(m)$ and of $m$. Since the top part of diagram~\eqref{diag: def of h*} is a pullback, so are the following downward commutative squares
\[
\xymatrix@=25pt{ Z\times_Y X' \pullback \ar@{=}[d] \ar[r] & \Eq(m) \ar@<-5pt>[d]_-{m_1} \ar@<5pt>[d]^-{m_2} \\ Z\times_Y X' \ar[r]_-{\pi_{X'}} & X'\ar[u]|-{s_0},}
\]
where $s_0=\la 1_{X'},1_{X'}\ra$. Consequently, the upward square is also a pullback. We get
\[
\xymatrix@R=30pt@C=50pt{Z\times_Y X'  \ar[rr]^-{\pi_{X'}}  \ar@{=}[dr]^-{h^*(s_0)=1_{Z\times_Y X'}} \ar@/_/@<-7pt>[ddr]_-{\pi'_Z} & & X' \ar@/_/@<-2pt>[ddr]_(.25){f'} \ar@{ >->}[dr]^-{s_0}  \\
		& Z\times_Y X' \ar[rr] \ar@<-2pt>[d]_-{\pi'_Z} & & \Eq(m) \ar@<-2pt>[d]_(.3){f'm_1} \\
		& Z \ar@<-2pt>[u]_-{i'_Z} \ar[rr]_-h \ar@/^/@<3pt>[uul]_-{i'_Z} & & Y. \ar@<-2pt>[u]_(.7){s_0s'} \ar@/^/@<-2pt>[uul]_(.65){s'}}
\]
Since $s_0$ is a split monomorphism such that $h^*(s_0)$ is an isomorphism, then $s_0$ is an isomorphism, by assumption. It easily follows that each $m_i$ is a isomorphism; thus $m$ is a monomorphism.
\end{proof}

\begin{lemma}\label{lem: h* reflects monos iff pbs reflect monos} In a finitely complete category $\C$, the following conditions are equivalent:
\begin{itemize}
	\item[(i)] any $h^*$ reflects monomorphisms;
	\item[(ii)] pullbacks in $\C$ reflect monomorphisms.
\end{itemize}
\end{lemma}
\begin{proof}
(ii) $\Rightarrow$ (i) Given diagram~\eqref{diag: def of h*} with $h^*(m)$ a monomorphism, we conclude that $m$ is a monomorphism since the top part of this diagram is a pullback in $\C$.\\
(i) $\Rightarrow$ (ii) Consider the following diagram where the bottom square is a pullback in $\C$, $m$ is a monomorphism and the top is obtain by taking vertical kernel pairs
\[
	\xymatrix@=25pt{ A \ar@{=}[d]  \ar[r] & \Eq(n) \ar@<-5pt>[d]_-{n_1} \ar@<5pt>[d]^-{n_2} \\ A \ar@{ >->}[d]_-m \ar[r] \pullback & C \ar[d]^-n \ar[u]|-{s_0} \\ B \ar[r] & D;}
\]
here $s_0=\la 1_C,1_C\ra$. It follows that the downward top squares are pullbacks. This implies that $n_1\colon (n_1,s_0)\to (1_C,1_C)$ is a morphism in $\Pt_C(\C)$ such that $f^*(n_1)=1_A$. By assumption, $n_1$ is a monomorphism, which implies it is an isomorphism in $\C$ (since it is also a split epimorphism). It easily follows that $n$ is a monomorphism.
\end{proof}

\begin{corollary}\label{cor: pointed + Spl-proto => (ker f=0 => f mono)}
Let $\C$ be a pointed $\Spl$-protomodular category. If $\ker(n\colon C\to D)=0$, then $n$ is a monomorphism.
\end{corollary}
\begin{proof} This easily follows by combining Proposition~\ref{prop: Spl-proto => h* reflects monos} and Lemma~\ref{lem: h* reflects monos iff pbs reflect monos}.
\end{proof}

\begin{remark}
It is well-known that in every protomodular category pullbacks reflect monomorphisms~\cite{NEKEAC}. We do not know whether the equivalent properties of Lemma~\ref{lem: h* reflects monos iff pbs reflect monos} characterise $\Spl$-protomodular categories.
\end{remark}

\begin{example}\label{ex: Mon is not Spl-proto} $\Mon$ is not a $\Spl$-protomodular category since there exist monoid homomorphisms with trivial kernel that are not monomorphisms. So, we find pullbacks that do not reflect monomorphisms (see Corollary~\ref{cor: pointed + Spl-proto => (ker f=0 => f mono)}).
\end{example}

A category is called \defn{semi-abelian}~\cite{SAC} when it is pointed, finitely (co)complete, Barr-exact~\cite{Barr} and protomodular. Let $\C$ be a semi-abelian category and $\M$ a class of monomorphisms in $\C^\op$ such that all morphisms in $\M^\op$ are cokernels. By Proposition~\ref{prop: M-proto iff orthogonal to M}, $\C^\op$ is $\M$-protomodular if and only if for any pushout diagram in $\C$
\begin{equation}\label{diag: po of split mono}
	\xymatrix@=30pt{A \ar[r]^-h \ar@{ >->}@<-2pt>[d]_-\alpha & C \ar@{ >->}@<-2pt>[d]_-{\alpha'} \\ B  \ar[r]_-u \ar@{>>}@<-2pt>[u]_-\beta & D,
		\ophalfsplitpushout \ar@{>>}@<-2pt>[u]_-{\beta'} }
\end{equation}
where $\alpha$ is a split monomorphism with retraction $\beta$, the pair of morphisms $(u,\beta)$ is orthogonal to the class $\M^\op$. This means that: if $(u,\beta)$ factors through a morphism $e\in \M^\op$ as in the commutative diagram
\[
	\xymatrix@=25pt{D & B \ar[l]_-u \ar@{>>}[r]^-\beta \ar@{-|>}[d]|-{e\in \M^\op} & A, \\& E \ar[ul]^a \ar[ur]_-b}
\]
then $e$ is an isomorphism. Since $e$ is a cokernel, it must be the cokernel of its kernel; thus the vertical arrows in the following diagram give a short exact sequence. By taking the kernels of $\beta, e$ and $u$, there exists a factorisation as above if and only if there exist the induced morphisms $\overline{a}, \overline{b}$
\begin{equation}\label{diag: main in SAC}
\vcenter{\xymatrix@C=10pt@R=15pt{
	& & \Ker(e) \ar@{|>->}[dd]|(.4){\ker(e)}  \ar@{.>}[dl]_-{\overline{b}} \ar@{.>}[dr]^-{\overline{a}} \\
	& \Ker(\beta) \ar@{|>->}[dr]|-{\ker(\beta)}  & & \Ker(u) \ar@{|>->}[dl]|-{\ker(u)} \\
	D & & B \ar[ll]_-u \ar@{>>}[rr]^-\beta \ar@{-|>}[dd]|-{e\in \M^\op} & & A.\\\\
	& & E \ar@{.>}[uull]^-a \ar@{.>}[uurr]_-b}}
\end{equation}

\begin{lemma}\label{lem: M-proto for Nil_c} Let $\C$ be a semi-abelian category and $\M$ a class of monomorphisms in $\C^\op$ such that all morphisms in $\M^\op$ are cokernels. Then $\C^\op$ is $\M$-protomodular if and only if, for any pushout of type~\eqref{diag: po of split mono} and cokernel $e\in \M^\op$, we have
\begin{equation}\label{diag: M-proto in C^op with kernels}
 \left(\, \Ker(e) \subseteq \Ker(\beta)\cap \Ker(u) \,\right) \;\;\To\;\; \Ker(e)=0.
\end{equation}
\end{lemma}
\begin{proof} Immediate from diagram~\eqref{diag: main in SAC} and the fact that $\Ker(e)=0$ implies that (the cokernel) $e$ is a monomorphism, since $\C$ is a protomodular category. Thus, $e$ is an isomorphism.
\end{proof}

Given a group $G$, we consider a descending chain of normal subgroups of $G$
\[
	G= \gamma_1(G) \supseteq \gamma_2(G) \supseteq \cdots
\]
where $\gamma_1(G)=G$, $\gamma_2(G)=[G,G]=\{xyx^{-1}y^{-1}: x,y\in G\}$, $\gamma_3(G)=[\gamma_2(G),G]$, $\dots$. Given a natural number $c\ge 1$, we say that $G$ is a \defn{nilpotent group of class $c$} if $\gamma_{c+1}(G)=0$. A nilpotent group of class 1 is just an abelian group. We fix a natural number $c\ge 2$ and denote by $\Nil_c$ the variety of nilpotent groups of class (at most) $c$. It is easy to check that $\Nil_c$ is a semi-abelian variety such that $\Ab\subseteq \Nil_c\subseteq \Grp$. As a consequence, the abelianisation functor for $\Nil_c$ is defined the same way as for $\Grp$
\[
	\begin{array}{rcl}
		(\;\;)^{\ab}: \Nil_c & \longrightarrow & \Ab. \\
		G & \longmapsto & G^\ab=G/[G,G]
	\end{array}
\]
The reflection morphism with respect to abelianisation is given by a surjective homomorphism denoted by $\eta_G\colon G\rightarrow G^\ab$, for any $G$ in $\Nil_c$.

\begin{lemma}\label{lem: Ker(beta) and Ker(u) in [B,B]} Under the conditions of Lemma~\ref{lem: M-proto for Nil_c} applied to $\C=\Grp$ or $\C=\Nil_c$, we have $\Ker(\beta)\cap \Ker(u)\subseteq [B,B]$.
\end{lemma}
\begin{proof} We give the proof for $\C=\Grp$; the proof for $\C=\Nil_c$ is similar. The abelianisation functor preserves the pushout~\eqref{diag: po of split mono} (since it is a left adjoint), so that
\[
	\xymatrix@=35pt{A^\ab \ar[r]^-{h^\ab} \ar@{ >->}@<-2pt>[d]_-{\alpha^\ab} & C^\ab \ar@{ >->}@<-2pt>[d]_-{\alpha'^\ab} \\ B^\ab  \ar[r]_-{u^\ab} \ar@{>>}@<-2pt>[u]_-{\beta^\ab} & D^\ab,
		\ophalfsplitpushout \ar@{>>}@<-2pt>[u]_-{\beta'^\ab} }
\]
is a pushout in $\Ab$. It is easy to check that the kernel of $\beta^\ab$ is isomorphic to the reflection $\eta_B$ restricted to $\Ker(\beta)$; call it $K$. This leads to the isomorphic split short exact sequence in $\Ab$
\[
	\xymatrix@=25pt{ \Ker(\beta^\ab) \ar@{|>->}[r]^-{\ker(\beta^\ab)} \ar@{=}[d] & B^\ab \ar@<2pt>@{>>}[r]^-{\beta^\ab} \ar[d]|-\cong & A^\ab \ar@<2pt>@{ >->}[l]^-{\alpha^\ab} \ar@{=}[d]\\
	\Ker(\beta^\ab) \ar[d]|-{\cong} \ar@{|>->}[r]_-{\la 1,0\ra} & \Ker(\beta^\ab)\oplus A^\ab \ar[d]|-\cong \ar@<2pt>@{>>}[r]^-{\pi_2} & A^\ab \ar@{=}[d] \ar@<2pt>@{ >->}[l]^-{\la 0,1\ra} \\
	K \ar@{|>->}[r]_-{\la 1,0\ra} & K\oplus A^\ab \ar@<2pt>@{>>}[r]^-{\pi_2} & A^\ab.\ar@<2pt>@{ >->}[l]^-{\la 0,1\ra} }
\]
Then, the downward pushout from above may be simplified to
\[
	\xymatrix@R=25pt@C=35pt{A^\ab \ar[r]^-{h^\ab} \ar@{|>->}[d]_-{\la 0,1\ra} & C^\ab \ar@{|>->}[d]^-{\la 0,1\ra} \\ K\oplus A^\ab  \ar[r]_-{1_K\oplus h^\ab} & K\oplus C^\ab. \ophalfsplitpushout}
\]
Now we prove that $\Ker(\beta)\cap \Ker(u)\subseteq [B,B]$. Given $x\in \Ker(\beta)\cap \Ker(u)$, we have $\eta_B(x)\in K$. On the other hand, $u^\ab(\eta_B(x))=\eta_D(u(a))=0$. Using the fact that $u^\ab\cong 1_K\oplus h^\ab$, this last equality is equivalent to $1_K\oplus h^\ab(\eta_B(x),0)=(0,0)$, from which we conclude that $\eta_B(x)=0$; thus $x\in [B,B]$.
\end{proof}

\begin{example}\label{ex: Nil_c^op} We show that $\Nil_c^\op$ is $\Spl$-protomodular and $\PSpl$-protomodular, where $\PSpl$ denotes the class of protosplit monomorphisms, i.e. kernels of split epimorphisms. We also show that $\Nil_c^\op$ is not $\SMono$-protomodular. Since $\Nil_c$ is a semi-abelian variety, $\Nil_c^\op$ is finitely complete.\\
\noindent\textbf{$\Nil_c^\op$ is $\Spl$-protomodular}\\
In this case the morphisms $e\in\M^\op$ as in diagram~\eqref{diag: main in SAC} are split epimorphisms. In $\Nil_c$ (and also in $\Grp$) it is easy to check that a quotient $e$ admits a splitting if and only if $\Ker(e)$ admits a complement in $B$, i.e. a subgroup $H\le B$ such that $\Ker(e)\cap H=0$ and $B=\Ker(e)H$.  By Lemma~\ref{lem: M-proto for Nil_c}, we must prove that for any pushout of type~\eqref{diag: po of split mono} and split epimorphism $e\in \M^\op$, we have
\[
	\left( \,\Ker(e) \subseteq \Ker(\beta)\cap \Ker(u) \,\right) \;\;\To\;\; \Ker(e)=0.
\]
We know that $\Ker(e)$ admits a complement $H$ in $B$ and $\Ker(e) \subseteq [B,B]$ by Lemma~\ref{lem: Ker(beta) and Ker(u) in [B,B]}. We get
\[
	B=\Ker(e)H\subseteq [B,B]H\subseteq B \;\; \To\;\; B=[B,B]H \;\; \To\;\; H=B.
\]
The last implication only holds in $\Nil_c$. By induction on $c$ one can prove that, if $B$ is a nilpotent group of class $c$ and $H\le B$ such that $B=[B,B]H$, then $B=H$. To finish, $\Ker(e)=\Ker(e)\cap B=\Ker(e)\cap H=0$.

\noindent\textbf{$\Nil_c^\op$ is $\PSpl$-protomodular}\\
In this case $e$ is the cokernel of a split monomorphism $i$ (with retraction $r$), whose kernel $\la\la Q\ra\ra$ is the normal closure of $Q$ in $B$. We consider $i$ to be an inclusion
\[
	\xymatrix@R=15pt@C=35pt{ \la\la Q\ra\ra \ar@{ |>->}[rr]^-{\ker(e)}  && B \ar@{|>}[r]^-{e=\coker(i)} \ar@<-2pt>@{>>}[dl]_-r & B/\la\la Q \ra\ra.\\ & Q \ar@{.>}[ul] \ar@<-2pt>@{_{(}->}[ur]_-i}
\]
Suppose that $\la\la Q\ra\ra \subseteq \Ker(\beta)\cap \ker(u)$. By Lemma~\ref{lem: Ker(beta) and Ker(u) in [B,B]}, $\la\la Q\ra\ra\subseteq[B,B]$. Considering $Q$ a subgroup of $B$, we have
\[
	Q=r(Q)\subseteq r(\la\la Q\ra\ra) \subseteq r([B,B])=[r(B),r(B)]=[Q,Q]\subseteq Q.
\]
We conclude that $Q=[Q,Q]$, i.e. it is a perfect (nilpotent) group. Up to this point, our arguments hold both in $\Nil_c$ and in $\Grp$. If $
Q$ is a nilpotent group of class $c$, then $0=\gamma_{c+1}(Q)=\gamma_c(Q)=\cdots =\gamma_1(Q)=Q$; thus $\la\la Q\ra\ra=0$. \\
\noindent\textbf{$\Nil_c^\op$ is not $\SMono$-protomodular}\\
In this case $e$ is just a strong epimorphism, and taking its kernel gives a short exact sequence (with no extra conditions on the basic setting of diagram~\eqref{diag: main in SAC}). In particular, we can consider a pushout~\eqref{diag: po of split mono} where $u$ is a strong epimorphism and take $e\colon \xymatrix@=10pt{B \ar@{|>}[r] &  A\, {}^B\!+ D}$ to be the diagonal (strong epimorphism) of the pushout of $(u,\beta)$. This implies that $\Ker(e)=\Ker(\beta)\cap\Ker(u)$. If $\Nil_c^\op$ is $\SMono$-protomodular, then $\Ker(e)=0$. However, this is not true in general. For example, in the pushout~\eqref{diag: po of split mono}, take $B=D_4=\la r,s: r^4=s^2=1, srs=r^{-1}\ra$ the diedral group of order 8 (which is nilpotent of class 2), $A=\la s\ra$, $C=0$ and $\alpha$ the inclusion. In this setting, $u$ is the cokernel of $\alpha$, so that $u\colon D_4\rightarrow D_4/\la\la s\ra\ra$, with $\la\la s\ra \ra=\{1,s,r^2,r^2s\}$. Consider the normal subgroup $L=\{1,r^2,rs,r^3s\}$ (because it has order 4 in a group of order 8) and define $\beta\colon D_4\to \la s\ra$ by $\beta(L)=\{1\}$ and $\beta(D_4\backslash L)=\{s\}$. We have $\Ker(\beta)\cap \Ker(u)=L\cap \la\la s\ra\ra=\{1,r^2\}\neq 0$.
\end{example}

\begin{example}\label{ex: Grp^op} $\Grp^\op$ is not $\SMono$-protomodular, since neither is $\Nil_c$. Although Lemma~\ref{lem: Ker(beta) and Ker(u) in [B,B]} and some of the arguments of Example~\ref{ex: Nil_c^op} hold for $\Nil_c$ and $\Grp$, we want to emphasize that $\Grp^\op$ is not $\Spl$-protomodular nor $\PSpl$-protomodular. \\
\noindent\textbf{$\Grp^\op$ is not $\Spl$-protomodular}\\
In the pushout~\eqref{diag: po of split mono} take $B=S_3$, $A=\la (1\, 2)\ra$, $C=0$, $\alpha$ the inclusion and $\beta(A_3)=\{\id\}$, $\beta(S_3\backslash A_3)=\{(12)\}$, where $A_3=\{\id, (123),(132)\}$ is the alternating group of order 3 (i.e. of even permutations). In this setting, $u$ is the cokernel of the split monomorphism $\alpha$. Since $\la\la (1\,2)\ra\ra=S_3$, then $u\colon S_3\to S_3/\la\la (1\,2)\ra\ra=0$. In diagram~\eqref{diag: main in SAC}, take $e=\beta$, which is a split epimorphism in $\Grp$. We have, $\Ker(e)=\Ker(\beta)=A_3$, $\Ker(u)=S_3$. So, $\Ker(e)\subseteq \Ker(\beta)\cap \Ker(u)$, but $\Ker(e)=A_3\neq 0$.\\
\noindent\textbf{$\Grp^\op$ is not $\PSpl$-protomodular}\\
Let $Q$ be a non-trivial perfect group, i.e. $0\neq Q=[Q,Q]$. In the pushout~\eqref{diag: po of split mono} take $B=Q\times Q$, $A=\Delta_Q=\{(q,q):q\in Q\}\cong Q$, $C=0$, $\alpha$ to be the inclusion and $\beta(q,q')=(q',q')$. In this setting, $u$ is the cokernel of the split monomorphism $\alpha$. Since $\la\la \Delta_Q\ra\ra=Q\times Q$ (because $Q$ is perfect), then $u\colon Q\times Q\to Q\times Q/\la\la \Delta_Q\ra\ra=0$. In diagram~\eqref{diag: main in SAC}, take $e$ to be the cokernel of the split monomorphism given by the inclusion $Q\times 0 \hookrightarrow Q\times Q$, so that $e\colon Q\times Q \twoheadrightarrow Q\times Q/Q\times 0\cong Q\cong \Delta_Q$. We have, $\Ker(e)=\Ker(\beta)= Q\times 0\cong Q$, $\Ker(u)=Q\times Q$. So, $\Ker(e)\subseteq \Ker(\beta)\cap \Ker(u)$, but $\Ker(e)\cong Q\neq 0$.
\end{example}

\section{$\M$-protomodular categories vs. $\M$-Mal'tsev categories}\label{sec: M-proto vs M-Mal'tsev}

It is well-known from~\cite{MCFPO} that every protomodular category is necessarily a Mal'tsev category. Recall from~\cite{CLP,CPP,CKP} that a \defn{Mal'tsev category} $\C$ is defined through the property that every reflexive relation $R\colon X\relto X$ in $\C$ is an equivalence relation.  An equivalent definition is to ask that any relation $D\colon X\relto Y$ in $\C$ is \defn{difunctional}. The original definitions asks for the base category $\C$ to be regular or just finitely complete. However, these properties on relations can be stated through generalised elements without any assumption on $\C$. In this work we shall focus on the Mal'tsev property concerning difunctionality of relations.

\begin{definition}\label{def of difunctional}
A relation $X\stackrel{d_1}{\longleftarrow} D \stackrel{d_2}{\longrightarrow} Y$ in $\C$ is called \emph{difunctional} when the relation
\begin{equation}\label{W-difunctional}
  \xymatrix@C=40pt{\C(W,X) & \ar[l]_-{\C(W,d_1)} \C(W,D) \ar[r]^-{\C(W,d_2)} & \C(W,Y)}
\end{equation}
in $\Set$ is difunctional~\cite{Riguet}, for every object $W$ of $\C$. More precisely, given morphisms $x,u\colon W\to X$, $y,v\colon W\to Y$, we have $(xD y \wedge uDy \wedge uDv)\Rightarrow xDv$, that is
\begin{equation}\label{eq: difunctional with relations}
\vcenter{
\hbox{\resizebox{0.9\textwidth}{!}{$
\begin{array}{llll}
	\xymatrix@C=10pt@R=10pt{& X \\ W \ar[ur]^-x \ar[dr]_-y \ar@{.>}[rr]|-{\la x,y \ra} & & D, \ar[ul]_-{d_1} \ar[dl]^-{d_2} \\ & Y} &
	\xymatrix@C=10pt@R=10pt{& X \\ W \ar[ur]^-u \ar[dr]_-y \ar@{.>}[rr]|-{\la u,y \ra} & & D, \ar[ul]_-{d_1} \ar[dl]^-{d_2} \\ & Y} &
	\xymatrix@C=10pt@R=10pt{& X \\ W \ar[ur]^-u \ar[dr]_-v \ar@{.>}[rr]|-{\la u,v \ra} & & D\;\;\Rightarrow \ar[ul]_-{d_1} \ar[dl]^-{d_2} \\ & Y}&
	\xymatrix@C=10pt@R=10pt{& X \\ W \ar[ur]^-x \ar[dr]_-v \ar@{.>}[rr]|-{\la x,v \ra} & & D. \ar[ul]_-{d_1} \ar[dl]^-{d_2} \\ & Y}
\end{array}$}}}
\end{equation}
\end{definition}

Given a category with binary products and a class $\M$ of monomorphisms, we say that a relation $X\stackrel{d_1}{\longleftarrow} D \stackrel{d_2}{\longrightarrow} Y$ is in $\M$ when the monomorphism $\la d_1,d_2\ra\colon D\mono X\times Y$ is in $\M$.

\begin{definition}\label{def: M-Mal'tsev} Let $\C$ be a category with binary products and $\M$ a class of monomorphisms of $\C$. We call $\C$ an \defn{$\M$-Mal'tsev category} when any relation $X\stackrel{d_1}{\longleftarrow} D \stackrel{d_2}{\longrightarrow} Y$ in $\M$ is difunctional.
\end{definition}

\begin{example}\label{ex: M-Mal'tsev cats}
\begin{itemize}
	\item[1.] When $\M=\Mono$, an $\M$-Mal'tsev category is just an ``ordinary'' Mal'tsev category.
	\item[2.] When $\M=\SMono$, an $\M$-Mal'tsev category is a weakly Mal'tsev category~\cite{JMF} (see definition in Remark~\ref{rem: proto => Mal'tsev}).
\end{itemize}
\end{example}

We show next that, if a finitely complete category $\C$ is $\M$-protomodular, where $\M$ is a class of monomorphisms (with certain properties), then $\C$ is an $\M$-Mal'tsev category. The proofs of the $\M$-versions in the following results are similar to those for ``ordinary'' protomodular and Mal'tsev categories.
\begin{proposition}\label{prop: M-proto => Pt_Y(CC) is unital} Let $\C$ be an $\M$-protomodular category, where $\M$ is a class of monomorphisms of $\C$. Given an arbitrary pullback of points
	\begin{equation}\label{diag:pb of points}
	\vcenter{\xymatrix@=25pt{ Z\times_Y X \ar@<-2pt>[d]_-{\pi_Z} \ar@<2pt>[r]^-{\pi_X} \pullback & X \ar@<-2pt>[d]_-f \ar@<2pt>[l]^-{i_X} \\
							   Z \ar@<-2pt>[u]_-{i_Z} \ar@<2pt>[r]^-h & Y, \ar@<-2pt>[u]_-s \ar@<2pt>[l]^-t}}
	\end{equation}
the pair $(i_Z,i_X)$ is orthogonal to the class $\M$.
\end{proposition}
\begin{proof} Since the total and right diagrams in
\[
	\xymatrix@=25pt{X \ar[r]^-{i_X} \ar@<-2pt>[d]_-f \ar@(u,u)[rr]^-{1_X} & Z\times_Y X \ar@<-2pt>[d]_-{\pi_Z} \ar[r]^-{\pi_X} & X\ar@<-2pt>[d]_-f \\
		Y \ar@<-2pt>[u]_-s \ar[r]_-t \ar@(d,d)[rr]_-{1_Y} & Z \ar@<-2pt>[u]_-{i_Z} \ar[r]_-h & Y \ar@<-2pt>[u]_-s}
\]
are pullbacks, then so is the left diagram. The result follows by Proposition~\ref{prop: M-proto iff orthogonal to M}.
\end{proof}

\begin{proposition}\label{prop: Pt_Y(CC) unital <=> M-Mal'tsev} Let $\C$ be a category with pullbacks and $\M$ a class of monomorphisms of $\C$ with the properties
\begin{itemize}
	\item[(P2)] $\M$ is closed under composition;
	\item[(P4)] for any pullback projections $\pi_Z$, $\pi_X$, the monomorphism $\la \pi_Z,\pi_X\ra\colon Z\times_Y X\mono Z\times X$ is in $\M$;
	\item[(P5)] $\M$ is stable under pullbacks.
\end{itemize}
The following conditions are equivalent:
\begin{itemize}
	\item[(i)] $\C$ satisfies the property of Proposition~\ref{prop: M-proto => Pt_Y(CC) is unital};
	\item[(ii)] any relation $\la d_1,d_2\ra\colon  D \mono X\times Y$ in $\M$ is difunctional.
\end{itemize}
\end{proposition}
\begin{proof} (i) $\Rightarrow$ (ii) Consider the relations as in~\eqref{eq: difunctional with relations}. We take the following pullbacks, and induced morphism $\lambda\colon W\to P$
\[
\xymatrix@=25pt{ W \ar@{.>}[dr]^-{\lambda} \ar@(r,u)[drr]^-{\la \la u,y\ra, \la u,v\ra \ra} \ar@(d,l)[ddr]_-{\la \la x,y \ra, \la u,y\ra \ra}\\
	& P \ar@{}[dr]|-{\caixa{A}} \pullback \ar@<-2pt>[d]_-{p_1} \ar@<2pt>[r]^-{p_2} & \Eq(d_1) \pullback\ar@<-2pt>[d]_-{d_{11}} \ar@<2pt>[r]^-{d_{12}} \ar@<2pt>[l]^-{i_2} & D \ar[d]^-{d_1} \ar@<2pt>[l]^-{s_1}\\
	& \Eq(d_2) \pullback \ar@<-2pt>[d]_-{d_{21}} \ar@<2pt>[r]^-{d_{22}} \ar@<-2pt>[u]_-{i_1}  & D \ar[r]_-{d_1} \ar[d]^-{d_2} \ar@<-2pt>[u]_-{s_1} \ar@<2pt>[l]^-{s_2} & X \\
	& D \ar@<-2pt>[u]_-{s_2} \ar[r]_-{d_2} & Y.
}
\]Next we take the pullback of $\la d_1,d_2\ra$ along $\la d_1 d_{21}p_1, d_2d_{12}p_2\ra$, which is a monomorphism in $\M$ by property (P5), and the induced morphisms $\alpha\colon \Eq(d_2)\to Q$, $\beta\colon \Eq(d_1)\to Q$
\[
	\xymatrix@R=25pt@C=30pt{\Eq(d_2) \ar@{.>}[dr]^-{\alpha} \ar@(r,u)[drrr]^-{d_{21}} \ar@(d,l)[ddr]_-{i_1} \\
		& Q \pullback \ar[rr]^-{q} \ar@{ >->}[d]_-{\mu\in \M} && D \ar@{ >->}[d]^-{\la d_1,d_2\ra\in \M} \\
		& P \ar[rr]_-{\la d_1 d_{21}p_1, d_2d_{12}p_2\ra} && X\times Y
	}
		\xymatrix@R=25pt@C=30pt{\Eq(d_1) \ar@{.>}[dr]^-{\beta} \ar@(r,u)[drrr]^-{d_{12}} \ar@(d,l)[ddr]_-{i_2} \\
		& Q \pullback \ar[rr]^-{q} \ar@{ >->}[d]_-{\mu\in \M} && D \ar@{ >->}[d]^-{\la d_1,d_2\ra\in \M} \\
		& P \ar[rr]_-{\la d_1 d_{21}p_1, d_2d_{12}p_2\ra} && X\times Y.
	}
\]
Since \caixa{A} is a pullback of points, then $(i_1,i_2)$ is orthogonal to $\M$ by assumption. From the commutativity of
\[
	\xymatrix@=25pt{ & Q \ar@{ >->}[d]|-{\mu\in \M} \\
		\Eq(d_2) \ar[ur]^-{\alpha}\ar[r]_-{i_1} & P & \Eq(d_1) \ar[l]^-{i_2} \ar[lu]_-{\beta}}
\]
we conclude that $\mu$ is an isomorphism. We then have $\la d_1,d_2\ra q\mu^{-1}\lambda = \la x,v\ra$, which proves that $xDv$.\\
(ii) $\Rightarrow$ (i) Consider the pullback of points below and a monomorphism $m\in\M$ such that $m\alpha=i_Z$ and $m\beta=i_X$ in
\[
	\xymatrix@=25pt{M \ar@{ >->}[dr]^-{m\in \M}   \\
		& Z\times_Y X \ar@<-2pt>[d]_-{\pi_Z} \ar@<2pt>[r]^-{\pi_X} \pullback & X \ar@(u,r)[ull]_-{\beta} \ar@<-2pt>[d]_-f \ar@<2pt>[l]^-{i_X} \\
		& Z \ar@(l,d)[uul]^-{\alpha} \ar@<-2pt>[u]_-{i_Z} \ar@<2pt>[r]^-h & Y, \ar@<-2pt>[u]_-s \ar@<2pt>[l]^-t}
\]
By (P2) and (P4) it follows that $\la \pi_Zm,\pi_Xm\ra\colon M\mono Z\times X$ is in $\M$; let us denote this relation by $M$. It is easy to check that
\[
\begin{array}{lll}
	\la \pi_Z m, \pi_X m\ra \alpha \pi_Z= \la \pi_Z,sh\pi_Z \ra & \To & \;\;\;\;\pi_Z \,M\, sh\pi_Z, \\
	\la \pi_Z m, \pi_X m\ra \alpha t f \pi_X= \la t f \pi_X, sh\pi_Z\ra & \To & tf \pi_X \,M\, sh\pi_Z, \\
	\la \pi_Z m, \pi_X m\ra \beta \pi_X= \la tf\pi_X,\pi_X \ra & \To & tf \pi_X \,M\, \pi_X.
\end{array}
\]
By difunctionality of $M$ we get $\pi_Z \, M\, \pi_X$, from which we conclude that $m$ is an isomorphism.
\end{proof}

\begin{proposition}\label{prop: protovsMaltsev} Let $\C$ be a category with pullbacks and $\M$ a class of monomorphisms of $\C$ with the properties
\begin{itemize}
	\item[(P2)] $\M$ is closed under composition;
	\item[(P4)] for any pullback projections $\pi_Z$, $\pi_X$, the monomorphism $\la \pi_Z,\pi_X\ra\colon Z\times_Y X\mono Z\times X$ is in $\M$;
	\item[(P5)] $\M$ is stable under pullbacks.
\end{itemize}If $\C$ is an $\M$-protomodular category then it is an $\M$-Mal'tsev category.
\end{proposition}
\begin{proof} This follows immediately from Propositions~\ref{prop: M-proto => Pt_Y(CC) is unital} and \ref{prop: Pt_Y(CC) unital <=> M-Mal'tsev}.
\end{proof}

\begin{remark}\label{rem: proto => Mal'tsev}
\begin{itemize}
	\item[1.] When $\M=\Mor$ or $\M=\Mono$, Proposition~\ref{prop: M-proto => Pt_Y(CC) is unital} translates to: any protomodular category is a Mal'tsev category. Indeed, any pullback of points~\eqref{diag:pb of points} being such that $(i_Z,i_X)$ is jointly extremally epimorphic is equivalent to every $\Pt_Y(\C)$ being a unital category~\cite{MCFPO}, for any object $Y$ of $\C$. This latter is a well-known characterisation of Mal'tsev categories~\cite[Proposition 10]{MCFPO}. Proposition~\ref{prop: Pt_Y(CC) unital <=> M-Mal'tsev} recovers the equivalence between two well-known characterisations of Mal'tsev categories: every $\Pt_Y(\C)$ is a unital category, for any object $Y$ of $\C$, and every relation $D\colon X\relto Y$ is difunctional.
	\item[2.] When $\M=\SMono$, Proposition~\ref{prop: M-proto => Pt_Y(CC) is unital} translates to: any weakly protomodular category is a weakly Mal'tsev category. Recall  from~\cite{NMFweakly} that a \emph{weakly Mal'tsev category} is defined by the property that any pullback of points~\eqref{diag:pb of points} is such that $(i_Z,i_X)$ is jointly epimorphic. Proposition~\ref{prop: Pt_Y(CC) unital <=> M-Mal'tsev} recovers the characterisation of weakly Mal'tsev categories as the finitely complete categories such that every strong relation $D\colon X\relto Y$ is difunctional (see~\cite{JMF}).
\end{itemize}
\end{remark}

\section{$\Ord$-protomodular categories}\label{sect:Ordproto}
In this section we explore the notion of $\M$-protomodular category when the base category $\CC$ is $\Ord$-enriched and $\M$ is the class of all full (and faithful) monomorphisms -- denoted by \defn{ff-monomorphisms} -- of $\CC$.

\begin{definition}\label{def: Ord-proto}
Let $\CC$ be a finitely complete $\Ord$-category. We call $\CC$ an \defn{$\Ord$-protomodular category} when, for any morphism $h\colon Z\to Y$, the change-of-base functor $h^*\colon \Pt_Y(\CC)\to \Pt_Z(\CC)$ is conservative on the class of ff-monomorphisms of $\CC$.
\end{definition}

\begin{lemma}\label{lem: proto => Ord-proto} If $\C$ is a protomodular category, then any $\Ord$-enrichment of $\C$ is an $\Ord$-pro\-to\-mo\-du\-lar category.
\end{lemma}
\begin{proof} As mentioned earlier, $\C$ is protomodular if and only if any change-of-base functor is conservative on monomorphisms. In particular, they are conservative on ff-monomorphisms.
\end{proof}

\begin{remark} Propositions~\ref{prop: M-proto iff SS5L for M}, \ref{prop: M-proto iff orthogonal to M} and \ref{prop: pb pp} hold in the context of $\Ord$-categories, where $\M$ is the class of ff-monomorphisms. Propositions~\ref{prop: pb pp for regular cat} and~\ref{prop: regular M-proto iff S5L for M} hold for any (pointed) regular $\Ord$-category with 2-pullbacks, so that 2-pullbacks and pullbacks coincide. Similarly, Proposition~\ref{prop: M-proto => Pt_Y(CC) is unital} holds for any $\Ord$-protomodular category and Proposition~\ref{prop: Pt_Y(CC) unital <=> M-Mal'tsev} for any $\Ord$-category with 2-pullbacks (again, so that 2-pullbacks and pullbacks coincide). Note that the assumptions (P1)--(P5) on morphisms needed for some of these results from Section~\ref{sec: M-protomodularity} and Section~\ref{sec: M-proto vs M-Mal'tsev} to hold, are replaced by the properties (F1)--(F5) on ff-monomorphisms listed below, which are easily seen to be always true in the context of $\Ord$-categories with 2-pullbacks.
\end{remark}
\begin{lemma}\label{pps for ff-ms}
Let $a\colon A\to B$ and $b\colon B\to C$ be morphisms in an $\Ord$-category $\CC$. Then:
\begin{itemize}
	\item[(F1)] if $ba$ is an ff-monomorphism, then $a$ is an ff-monomorphism;
	\item[(F2)] if $a$ and $b$ are ff-monomorphisms, then $ba$ is also an ff-monomorphism;
	\item[(F3)] if $a$ is an ff-monomorphism, then $a\times a$ is an ff-monomorphism;
	\item[(F4)] for any 2-pullback projections $\pi_Z$, $\pi_X$, the morphism $\la \pi_Z,\pi_X\ra\colon Z\times_Y X\mono Z\times X$ is an ff-monomorphism;
	\item[(F5)] ff-monomorphisms are stable under 2-pullbacks.\hfill $\Box$
\end{itemize}
\end{lemma}
The only properties on ff-monomorphisms concerning the $\Ord$-enrichment of $\CC$ are (F4) and (F5), where 2-pullbacks are mentioned instead of pullbacks. However, if the finitely complete category $\CC$ admits a 2-pullback of a pair of morphisms $(f,g)$, then it coincides with the pullback of $(f,g)$.

The notion of $\Ord$-Mal'tsev category introduced in \cite{CR} can also be seen as a particular instance of the notion of $\M$-Mal'tsev category, when $\M$ is the class of ff-monomorphisms. Hence Proposition~\ref{prop: protovsMaltsev}  gives, in particular, that every $\Ord$-protomodular category is $\Ord$-Mal'tsev.

\begin{remark}
The category $\Ord^\op$, with its natural $\Ord$-enrichment given, for morphisms $f,g\colon X\to Y$, by $f\le g$ if $f(y)\le g(y)$ for all $y\in Y$ is not $\Ord$-protomodular. As we have seen in Example \ref{exe:Ord}, $\Ord^\op$ is an $\SMono$-protomodular category which is not protomodular. In this enriched context, ff-monomorphisms coincide with monomorphisms (that is, surjective monotone maps in $\Ord$), not with strong monomorphisms, hence $\Ord$-protomodularity fails because protomodularity fails.
\end{remark}

\begin{example} {\bf $\Ord$-enrichments of $\OrdGrp$.}
We start by recalling two different $\Ord$-enrichment of $\OrdGrp$ studied in \cite{CMR}. We denote by $\OOrdGrp$ the one given, for any morphisms $f,g\colon X\to Y$, by $f\le g$ if, for all $x\in X$, $f(x)\le g(x)$, while we denote by $\OOrdGrp_+$ the other enrichment, where $f\le g$ if $f(x)\le g(x)$ for all positive elements $x$ of $X$. We point out that the former enrichment induces an equivalence relation on the morphisms.

As shown in \cite{CM-FM} and in Example \ref{exe:Grp} respectively, $\OrdGrp$ is not a protomodular category but it is $\SMono$-protomodular. Here we show that neither $\OOrdGrp$ nor $\OOrdGrp_+$ are $\Ord$-protomodular. The failure of this property is due to the fact that, for both $\Ord$-enrichments, ff-monomorphisms do not characterise ``initial monomorphisms", that is, injective homomorphisms whose domain has the initial topology. Indeed:
\begin{itemize}
\item a monomorphism $m\colon X'\to X$ is an ff-monomorphism in $\OOrdGrp$ if and only if, for all $x,x'\in X'$, $x\cong x'$ whenever $m(x)\cong m(x')$, so, in particular, if the order on $X'$ is anti-symmetric $m$ is an ff-monomorphism if and only if it is monomorphic;
\item a monomorphism $m\colon X'\to X$ is an ff-monomorphism in $\OOrdGrp_+$ if and only if, for all $x,x'\in X'$, if $0\le x$ and $0\le x'$, then $x\leq x'$ whenever $m(x)\le m(x')$. Hence if the only positive element of $X'$ is $0$, $m$ is an ff-monomorphism if and only if it is monomorphic.
\end{itemize}
Diagrams \eqref{diag:grp} and \eqref{diag:grpp} show, respectively, that $\OOrdGrp$ and $\OOrdGrp_+$ fail to satisfy the ff-monomorphism version of protomodularity.
\begin{equation}\label{diag:grp}
\vcenter{\xymatrix@=25pt{ (\ZZ,\NN) \ar@{ >->}[r]^-{\la 1_{\ZZ},0\ra} \ar@{=}[d] &  (\ZZ\times \ZZ, P_\prod) \ar@<2pt>[r]^-{\pi_2} \ar[d]_-{1_{\ZZ\times \ZZ}} & (\ZZ,\NN) \ar@<2pt>[l]^-{\la 0, 1_{\ZZ}\ra} \ar@{=}[d] \\
		(\ZZ,\NN) \ar@{ >->}[r]_-{\la 1_{\ZZ},0\ra} & (\ZZ\times \ZZ, P_\lex) \ar@<2pt>[r]^-{\pi_2} & (\ZZ,\NN), \ar@<2pt>[l]^-{\la 0, 1_{\ZZ}\ra}}}
\end{equation}
\begin{equation}\label{diag:grpp}
\vcenter{\xymatrix@=25pt{ (\ZZ,\{0\}) \ar@{ >->}[r]^-{\la 1_{\ZZ},0\ra} \ar@{=}[d] &  (\ZZ\times \ZZ, P_\prod) \ar@<2pt>[r]^-{\pi_2} \ar[d]_-{1_{\ZZ\times \ZZ}} & (\ZZ,\NN) \ar@<2pt>[l]^-{\la 0, 1_{\ZZ}\ra} \ar@{=}[d] \\
		(\ZZ,\{0\}) \ar@{ >->}[r]_-{\la 1_{\ZZ},0\ra} & (\ZZ\times \ZZ, P_\lex) \ar@<2pt>[r]^-{\pi_2} & (\ZZ,\NN), \ar@<2pt>[l]^-{\la 0, 1_{\ZZ}\ra}}}
\end{equation}
In both diagrams the order of $\bkZ\times\bkZ$ in the top row is the product order, while in the bottom row is the lexicographic order (see \cite{CM-FM} for details); that is, in \eqref{diag:grp} the positive cones are given by $P_\prod=\bkN\times\bkN$ and $P_\lex=\{(n,m)\in \bkZ\times\bkZ\,;\,m>0\}\,\cup\,\{(n,0)\,;\,n\in\bkN\}=\bkZ\times\bkN_1\,\cup\,\bkN\times\{0\}$, while in \eqref{diag:grpp} $P_\prod=\{0\}\times\bkN$ and $P_\lex=\bkZ\times\bkN_1\,\cup\,\{(0,0)\}$. In both cases $P_\prod \subset P_\lex$ and so the identity $1_{\bkZ\times\bkZ}$ is a monotone bijection but not an isomorphism.

\end{example}

\begin{remark}
\begin{enumerate}
\item[1.] As shown in \cite{CR}, $\OOrdGrp_+$ is $\Ord$-Mal'tsev, hence this is an example of an $\Ord$-Mal'tsev category which is not $\Ord$-protomodular.

\item[2.] We do not have an example of an $\Ord$-protomodular category which is not protomodular.

\item[3.] We point out that the notion of $\Ord$-protomodular is not directly related to the notion of lax protomodularity introduced in \cite{CMR}: lax protomodularity deals with the properties of the lax comma change-of-base functor.
\end{enumerate}
\end{remark}

\noindent\textbf{Disclosure statement}
Examples~\ref{ex: paratopological groups} and~\ref{ex: Nil_c^op} were obtained with the help of Anthropic's Claude Fable 5.1 (via Claude Code), in response to our direct questions whether there were such examples. We did not use any other help from AI to produce this paper.

\end{document}